\documentclass[11pt]{amsart}
\usepackage{amsmath,amsthm,amscd,amsfonts, amssymb, mathrsfs}
\usepackage{graphicx}
\usepackage{tikz}
\usepackage{color}
\include{diagxy}
\usepgflibrary{shapes}
\newcommand{\sect}[1]{\section{#1}\setcounter{equation}{0}}
\newcommand{\subsect}[1]{\subsection{#1}}

\font\mbn=msbm10 scaled \magstep1
\font\mbs=msbm7 scaled \magstep1
\font\mbss=msbm5 scaled \magstep1
\newfam\mbff
\textfont\mbff=\mbn
\scriptfont\mbff=\mbs
\scriptscriptfont\mbff=\mbss   

\newcommand{\N}       { \mathbb{N}}
\newcommand{\Z}        {\mathbb{Z}  } 
\newcommand\Co           {{\mathbb C}}

\newtheorem{Th}{Theorem}[section]
\newtheorem{Lm}[Th]{Lemma}
\newtheorem{C}[Th]{Corollary}

\newtheorem{Proposition}[Th]{Proposition}
\newtheorem{R}[Th]{Remark}

\newtheorem{E}[Th]{Example}

\begin{document}
\title[On Homotopy Invariants of Tensor Products of Banach Algebras]{On Homotopy Invariants of Tensor Products of Banach Algebras}

\author{Alexander Brudnyi} 
\address{Department of Mathematics and Statistics\newline
\hspace*{1em} University of Calgary\newline
\hspace*{1em} Calgary, Alberta\newline
\hspace*{1em} T2N 1N4}
\email{abrudnyi@ucalgary.ca}

\thanks{Research is supported in part by NSERC}

\maketitle

\begin{abstract}
We generalize results of Davie and Raeburn describing homotopy types of the group of invertible elements and of the set of idempotents of the projective tensor product of complex unital  Banach algebras. We illustrate our results by specific examples.
\end{abstract}

\sect{Formulation of Main Results}
\subsect{} We continue to study certain relations between the algebraic structure of a commutative Banach algebra and topological properties of its maximal ideal space. In this paper, in the framework of the Novodvorski--Taylor--Raeburnwe we generalize some results of Davie \cite{D} and Raeburn \cite{R} in this area and then use these generalizations to study homotopy invariants of important classes of Banach algebras including certain Wiener algebras and algebras of bounded holomorphic functions on Behrens domains. In particular, we obtain extensions of some results of \cite{Br4} on the completion problem to an invertible operator-valued function in uniform algebras and of \cite{BRS}, see Section~1.2 below.

To formulate our results we introduce some definitions and notations.

Let $\mathfrak A$ be a complex unital Banach algebra with center $Z(\mathfrak A)$. For unital subalgebras $A\subset Z(\mathfrak A)$ and $B\subset \mathfrak A$ we denote by $A\widehat\otimes_{\mathfrak A} B$ the closure of their algebraic tensor product  \penalty-10000 $A\otimes B\subset\mathfrak A$. Let $\mathfrak M(A)$ be the maximal ideal space of $A$, i.e., the set of nonzero homomorphisms $A\rightarrow\mathbb C$ equipped with the {\em Gelfand topology}. 
We consider the class of Banach algebras $\mathfrak A$ satisfying the following property. 

There exists a constant $C$ such that for all $\xi\in\mathfrak M(A)$ and every $n\in\N$ and  $a_i\in A$, $b_i\in B$, $1\le i\le n$,
\begin{equation}\label{eq1}
\left\|\sum_{k=1}^n\xi(a_k)b_k\right\|_{B}\le C\left\|\sum_{k=1}^n a_kb_k\right\|_{\mathfrak A}.
\end{equation}
Then the map $A\otimes B\rightarrow B$,
\begin{equation}\label{eq2}
\left( \sum_{k=1}^n a_kb_k  \right)\mapsto \sum_{k=1}^n\xi(a_k)b_k\in B,
\end{equation}
extends by continuity to a bounded multiplicative projection
$P_\xi:  A\widehat\otimes_{\mathfrak A}  B\rightarrow B$. 

Let $C(\mathfrak M(A),B)$ be the Banach algebra of $B$-valued continuous functions $f$ on $\mathfrak M(A)$ with norm $\|f\|:=\max_{\xi\in \mathfrak M(A)}\|f(\xi)\|_B$.
By $P: A\widehat\otimes_{\mathfrak A}  B\rightarrow C(\mathfrak M(A),B)$, $P(u)(\xi):=P_\xi (u)$, we denote the corresponding morphism of Banach algebras. (Note that if $B=\mathbb C$, then $P$ coincides with the Gelfand transform $\hat \,: A\rightarrow C(\mathfrak M(A),\mathbb C)$.) 

Further, for a complex Banach algebra $\mathcal A$ with unit $1$ we denote by $\mathcal A^{-1}$  the group of invertible elements.  Also, by $\mathcal A^{-1}_l=\{a\in\mathcal A\, :\ \exists\, b\in\mathcal A\ {\rm such\ that}\ ba=1\}$ we denote the set of left-invertible elements of $\mathcal A$ and by ${\rm id}\, \mathcal A=\{a\in \mathcal A\, :\, a^2=a\}$  the set of idempotents of $\mathcal A$. Then $\mathcal A^{-1}$ is a complex Banach Lie group,  $\mathcal A^{-1}$ and $\mathcal A^{-1}_l$ are open subsets of $\mathcal A$ and ${\rm id}\, \mathcal A$ is the discrete union of closed connected complex Banach submanifolds of $\mathcal A$.
Moreover, each connected component of $\mathcal A^{-1}_l$ is a complex Banach homogeneous space under the action of $\mathcal A_0^{-1}$ (-- the connected component of $\mathcal A^{-1}$ containing $1$) by left multiplications: $\mathcal A_0^{-1}\times  \mathcal A^{-1}_l \ni (g,a)\mapsto ga\in \mathcal A^{-1}_l$, see \cite[Prop.\,4.7]{Br4}. In turn, 
each connected component of ${\rm id}\, \mathcal A$  is a complex Banach homogeneous space under the action of $\mathcal A_0^{-1}$ by  similarity transformations: $\mathcal A^{-1}_0\times {\rm id}\, \mathcal A\ni(g,a)\mapsto gag^{-1}\in  {\rm id}\, \mathcal A$, see \cite[Cor.\,1.7]{R}. (For the definition of a complex Banach homogeneous space see, e.g., \cite[Sec.\,1]{R}.)

Note that $P$ maps topological spaces $(A\widehat\otimes_{\mathfrak A}  B)^{-1}$, $(A\widehat\otimes_{\mathfrak A}  B)_l^{-1}$ and ${\rm id}(A\widehat\otimes_{\mathfrak A}  B)$ continuously into the spaces $(C(\mathfrak M(A),B))^{-1}$,
$(C(\mathfrak M(A), B))_l^{-1}$ and ${\rm id}(C(\mathfrak M(A), B))$, respectively. 
Here
$(C(\mathfrak M(A),B))^{-1}=C(\mathfrak M(A),B^{-1})$, ${\rm id}(C(\mathfrak M(A), B))=C(\mathfrak M(A), {\rm id}\, B)$ and $C(\mathfrak M(A), B_l^{-1})=(C(\mathfrak M(A), B))_l^{-1}$ (due to Allan's theorem, see, e.g., \cite[Thm.\,9.2.7]{Ni}).

Recall that topological spaces $X$ and $Y$ are called homotopy equivalent if there exist continuous maps $f:X\rightarrow Y$ and $g:Y\rightarrow X$ such that 
$g\circ f$ and $f\circ g$ are homotopic to the identity maps ${\rm id}_X$ and ${\rm id}_Y$. Each of the maps $f$ and $g$ is called a homotopy equivalence.

Our main result reads as follows.
\begin{Th}\label{main}
Let $(X,Y)$ be one of the pairs $\bigl((A\widehat\otimes_{\mathfrak A}  B)^{-1}, (C(\mathfrak M(A),B))^{-1}\bigr)$, \penalty-10000
$\bigl((A\widehat\otimes_{\mathfrak A}  B)_l^{-1}, (C(\mathfrak M(A),B))_l^{-1}\bigr)$ or
$\bigl({\rm id}(A\widehat\otimes_{\mathfrak A}  B), {\rm id}(C(\mathfrak M(A), B))\bigr)$.
Then $P:X\rightarrow Y$ is a homotopy equivalence.
\end{Th}
\begin{R}\label{rem1}
{\rm (a) If $\mathfrak A=A\widehat{\otimes}_\pi B$ is the projective tensor product of $A$ and $B$, then Theorem \ref{main} implies Davie's theorem \cite[Thm. 4.10]{D} asserting that the Gelfand transform induces a homotopy equivalence of $(A\widehat{\otimes}_\pi B)^{-1}$ and $C(\mathfrak M(A), B^{-1})$.\smallskip

\noindent (b) For topological spaces $X$ and $Y$ let $[X]$ denote the set of connectivity components of 
$X$ and $[X, Y]$ denote the set of homotopy classes of continuous maps from $X$ to $Y$. Theorem \ref{main} implies that $P$ induces bijections between sets $[(A\widehat\otimes_{\mathfrak A}  B)^{-1}]$, $[(A\widehat\otimes_{\mathfrak A}  B)_l^{-1}]$,  $[{\rm id}(A\widehat\otimes_{\mathfrak A}  B)]$ and $[\mathfrak M(A), B^{-1}]$, $[\mathfrak M(A),B_l^{-1}]$, $[\mathfrak M(A),{\rm id}\, B]$, respectively. For $\mathfrak A=A\widehat{\otimes}_\pi B$ the latter gives Raeburn's theorem \cite[Thm.\,4.5]{R} asserting that the Gelfand transform induces a bijection $[{\rm id}(A\widehat{\otimes}_\pi B)]\rightarrow [\mathfrak M(A), {\rm id}\, B]$.\smallskip

\noindent (c) Let $\alpha$ be a norm on $A\otimes B$ such that the completion $\mathfrak A:=A\widehat\otimes_\alpha B$ of $A\otimes B$ with respect to $\alpha$ is a Banach algebra satisfying \eqref{eq1}.
Then Theorem \ref{main} asserts that the homotopy type of each of the considered objects is independent of the choice of such $\alpha$. In particular, if $A$ is a uniform Banach algebra (i.e., $\|a^2\|_A=\|a\|_A^2$ for all $a\in A$), the result can be applied to the injective tensor product $A\widehat\otimes_{\varepsilon}B$ as well (because  $A\widehat\otimes_{\varepsilon}B$ is isometrically isomorphic to a subalgebra of $C(\mathfrak M(A),B)$).}
\end{R}
Let $\mathcal A$ be a complex unital Banach algebra. We say that $a,b\in \mathcal A_l^{-1}$ are equivalent (written $a\sim b$) if $b=ga$ for some $g\in\mathcal A^{-1}$. Similarly, $a,b\in {\rm id}\,\mathcal A$ are equivalent (written $a\sim b$) if $b=gag^{-1}$ for some $g\in\mathcal A^{-1}$. Clearly, $\sim$ are equivalence relations on $\mathcal A_l^{-1}$ and ${\rm id}\,\mathcal A$. The corresponding sets of equivalence classes will be denoted by $\{\mathcal A_l^{-1}\}$ and $\{{\rm id}\,\mathcal A\}$.

In our setting, since $P:A\widehat\otimes_{\mathfrak A}  B\rightarrow C(\mathfrak M(A),B)$ is a morphism of Banach algebras, it induces maps $\{P\}_l:\{(A\widehat\otimes_{\mathfrak A}  B)_l^{-1}\}\rightarrow \{(C(\mathfrak M(A), B))_l^{-1}\}$ and $\{P\}_{id}:\{{\rm id}(A\widehat\otimes_{\mathfrak A}  B)\}\rightarrow \{{\rm id}(C(\mathfrak M(A), B))\}$.

As a corollary of Theorem \ref{main} we obtain:
\begin{C}\label{cor1}
Maps  $\{P\}_l$ and $\{P\}_{id}$ are bijections.
\end{C}
\subsect{} 
In this part we present some applications of Theorem \ref{main}. 

Let $\mathscr D$ denote the class of commutative unital complex Banach algebras whose maximal ideal spaces are inverse limits of inverse limiting systems of families of compact connected spaces with trivial second \v{C}ech cohomology groups homotopy equivalent to compact spaces of covering dimension $\le 2$. (Recall that for a normal space $X$, the covering dimension ${\rm dim}\, X \le n$ if every open cover of $X$ can be refined by an open cover whose order is $\le n+1$.) 

By $\mathscr C$ we denote the subclass of $\mathscr D$ consisting of algebras whose maximal ideal spaces are inverse limits of inverse limiting systems of families of compact contractible spaces. 

Examples of algebras in $\mathscr C$ include, e.g., the algebra $C(K)$ of continuous functions  on a compact contractible space $K$, the algebra $A(U)$ of bounded holomorphic functions continuous up to the boundary  on a convex domain $U\subset\mathbb C^n$,  the Wiener algebra $W(G)_\Sigma$ of functions on a connected compact abelian group $G$ whose Bohr--Fourier spectra are subsets of a fixed subsemigroup $\Sigma$ of the (additive) dual group containing zero and having the property that it does not contain both a nonzero element and its opposite and its analogue $C(G)_\Sigma$ consisting of continuous functions on $G$, see, e.g., \cite{BRS}.  In turn, the basic examples of algebras in $\mathscr D\setminus\mathscr C$ are the algebra $C(K)$ with $K$ being a compact space homotopy equivalent to a space of covering dimension $\le 2$ and with $H^2(K,\Z)=0$ but $H^1(K,\Z)\ne 0$, the algebra $A(K)$ -- the uniform closure in $C(K)$ of the algebra of  holomorphic functions defined on neighbourhoods of a compact polynomially convex set $K\subset\mathbb C^n$ homotopy equivalent to a space of covering dimension $\le 2$ and with $H^2(K,\Z)=0$ but $H^1(K,\Z)\ne 0$, algebras $H^\infty(R)$ of bounded holomorphic functions on Riemann surfaces $R$, where $R$ is an unbranched  covering of a bordered Riemann surface or a Behrens type planar domain constructed from such a covering, see \cite[Thm.\,1.1,\,Thm.\,1.3]{Br5}.\smallskip

The following result generalizes Theorems 1.1 and results of Section~5 of \cite{BRS}.
\begin{Th}\label{te1.4}
Suppose $A\in \mathscr C$. 
Let $X\in\bigl\{ (A\widehat\otimes_{\mathfrak A}  B)^{-1}, (A\widehat\otimes_{\mathfrak A}  B)_l^{-1}, {\rm id}(A\widehat\otimes_{\mathfrak A}  B)\bigr\}$ and $Y$ be the corresponding space from $\bigl\{B^{-1}, B_l^{-1}, {\rm id}\,B\bigr\}$. Let $\xi\in\mathfrak M(A)$. Then the map $P_\xi: X\rightarrow Y$ is a homotopy equivalence. Moreover, $P_\xi$ induces bijections between $\{(A\widehat\otimes_{\mathfrak A}  B)_l^{-1}\}$ and $\{B_l^{-1}\}$ and between $\{{\rm id}(A\widehat\otimes_{\mathfrak A}  B)\}$ and $\{{\rm id}\, B\}$.
\end{Th}
\begin{C}\label{cor1.5}
Suppose $A\in \mathscr C$.
\begin{itemize}
\item[(1)]
For every $u\in {\rm id}(A\widehat\otimes_\mathfrak A B)$ there exists  $g\in (A\widehat\otimes_\mathfrak A B)_0^{-1}$ such that $gug^{-1}=P_\xi(u)$.\smallskip
\item[(2)]
For every $u\in (A\widehat\otimes_\mathfrak A B)_l^{-1}$ there exists $g\in (A\widehat\otimes_\mathfrak A B)_0^{-1}$ such that $gu=P_\xi(u)$. 
\end{itemize}
\end{C}

Next, recall that if $X$ is a complex Banach homogeneous manifold under the action $G\times X\rightarrow X$, $(g,x)\mapsto g\bullet x$, of a complex Banach Lie group $G$, then the stabilizer $G(x)$ of a point  $x\in X$,
$G(x):=\{g\in G\, :\, g\bullet x=x\}$, is a closed complex Banach Lie subgroup of $G$ and stabilizers of different points are conjugate in $G$ by inner automorphisms, see, e.g., \cite[Prop.\,1.4]{R}.

Our next result gives a partial extension of \cite[Thm.\,4.4\,(b),\,Thm.\,4.8\,(a)]{Br5} (proved for the algebra $H^\infty(\mathbb D)$, where $\mathbb D\subset\mathbb C$ is the open unit disk).
\begin{Th}\label{teo4.1}
Suppose $A\in \mathscr D\setminus\mathscr C$. Let $\xi\in\mathfrak M(A)$. 
\begin{itemize}
\item[(1)] 
If the stabilizer  $B_0^{-1}(P_\xi(u))\subset  B_0^{-1}$ of  $P_\xi(u)\in {\rm id}\, B$ is connected, then for every $u\in {\rm id}(A\widehat\otimes_\mathfrak A B)$
there exists $g\in (A\widehat\otimes_\mathfrak A B)_0^{-1}$ such that $gug^{-1}=P_\xi(u)$. \smallskip
\item[(2)] 
If the stabilizer  $B_0^{-1}(P_\xi(u))\subset B_0^{-1}$ of  $P_\xi(u)\in B_l^{-1}$ is connected, then for every $u\in (A\widehat\otimes_\mathfrak A B)_l^{-1}$
there exists $g\in (A\widehat\otimes_\mathfrak A B)_0^{-1}$ such that $gu=P_\xi(u)$. 
\end{itemize}
\end{Th}
\begin{E}\label{ex4.3}
{\rm  Let $L(X)$ be the Banach algebra of bounded linear operators on a complex Banach space $X$ equipped with the operator norm. By $I_X\in L(X)$ we denote the identity operator and by $GL(X)\subset L(X)$ the set of invertible bounded linear operators on $X$. Clearly, $L(X)^{-1}:=GL(X)$.  By $GL_0(X)\subset GL(X)$ we denote the connected component of $I_X$.
Each $E\in {\rm id}\, L(X)$ determines a direct sum decomposition $X=X_0\oplus X_1$, where $X_0:={\rm ker}\,E$ and $X_1:={\rm ker}\,(I_X-E)$. It is easily seen that the stabilizer
$GL_0(X)(E)\subset GL_0(X)$ consists of all operators $B\in GL_0(X)$ such that $B(X_k)\subset X_k$, $k=1,2$. In particular, restrictions of operators in $GL_0(X)(E)$ to $X_k$ determine a monomorphism of complex Banach Lie groups $S_E: GL_0(X)(E)\rightarrow GL(X_1)\oplus GL(X_2)$. Moreover, $S_E$ is an isomorphism if $GL(X_i)$, $i=1,2$ are connected. 

Let $A\in \mathscr D\setminus\mathscr C$ be a uniform complex Banach algebra. Then  the morphism $P$ maps the injective tensor product $A\widehat{\otimes}_\varepsilon L(X)$ isometrically isomorphic onto a subalgebra of the algebra $C(\mathfrak M(A), L(X))$ of operator-valued continuous functions on $\mathfrak M(A)$. Let $\xi\in \mathfrak M(A)$. Now, Theorem \ref{teo4.1}\,(1) leads to the following statement:}\smallskip

\noindent $(A)$ Let $F\in {\rm id}(A\widehat{\otimes}_\varepsilon L(X))$ be such that
for $E:=P_\xi(F)\in {\rm id}\, L(X)$ the corresponding groups $GL(X_1)$ and $GL(X_2)$ are connected. Then  there exists $G\in (A\widehat{\otimes}_\varepsilon L(X))_0^{-1}$ such that $GFG^{-1}=E$.\smallskip

{\rm In particular, the result holds true for $X$ being one of the spaces: a finite-dimensional space, a Hilbert space, $c_0$ or $\ell^p$, $1\le p\le\infty$. Indeed, $GL(X)$ is connected if ${\rm dim}_{\Co}X<\infty$ and contractible for other spaces of the list, see, e.g., \cite{M} and references therein. Moreover, each subspace of a Hilbert space is Hilbert, and each
infinite-dimensional complemented subspace of $X$ being either $c_0$ or $\ell^p$, $1\le p\le\infty$, is isomorphic to $X$, see  \cite{Pe}, \cite{Lin}. This gives the required condition in (A). (It is worth noting that there are complex Banach spaces $X$ for which groups $GL(X)$ are not connected, see, e.g., \cite{Do}.) \smallskip

Further, each $E\in L(X)_l^{-1}$ determines 
a complemented subspace $X_1:={\rm ran}\,E\subset X$ isomorphic to $X$.
Then the stabilizer
$GL_0(X)(E)\subset GL_0(X)$ of $E$ consists of all operators $B\in GL_0(X)$ such that $B|_{X_1}=I_{X_1}$. If $X_2\subset X$ is a complemented subspace to $X_1$, then each $B\in GL_0(X)(E)$ has a form
\begin{equation}\label{equ4.1}
B=\left(
\begin{array}{ccc}
I_{X_1}&C\\
0&D
\end{array}
\right),\quad {\rm where}\quad D\in GL(X_2)\ \, {\rm and}\  \, C\in L(X_2,X_1)
\end{equation}
(here $L(X_2,X_1)$ is the Banach space of bounded linear operators $X_2\rightarrow X_1$).

\noindent Thus $GL_0(X)(E)$ is homotopy equivalent to the subgroup of $GL(X_2)\, (\cong GL(X/X_1))$ consisting of all operators $D$ such that ${\rm diag}(I_{X_1}, D)\in GL_0(X)$. In particular, this subgroup coincides with $GL(X_2)$ if the latter is connected.

Now, Theorem \ref{teo4.1}\,(2) leads to the following statement:\smallskip

\noindent $(B)$ {\em Let $F\in (A\widehat{\otimes}_\varepsilon L(X))_l^{-1}$ be such that for $E:=P_\xi(F)\in L(X)_l^{-1}$ the corresponding group $GL(X/X_1)$ is connected. Then there exists $G\in (A\widehat{\otimes}_\varepsilon L(X))_0^{-1}$ such that $F=GE$.}\smallskip

Let us identify $X$ with $X_1$  by  $E$ and regard $F$ as an element
of $A\widehat{\otimes}_\varepsilon L(X_1,X_1\oplus X_2)$ and $E$ as an operator in $L(X_1,X_1\oplus X_2)$.  Then we obtain
\[
F=
\left(
\begin{array}{ccc}
F_1\\
F_2
\end{array}
\right),\qquad
G=
\left(
\begin{array}{ccc}
G_{11}&G_{12}\\
G_{21}&G_{22}
\end{array}
\right)\quad {\rm and}\quad
E=
\left(
\begin{array}{cc}
I_{X_{1}}\\
0
\end{array}
\right),
\]
where $F_i\in A\widehat\otimes_\varepsilon L(X_1,X_i)$, $i=1,2$, and  $G_{ii}\in A\widehat\otimes_\varepsilon L(X_i)$, $i=1,2$,
$G_{ij}\in A\widehat\otimes_\varepsilon L(X_j,X_i)$, $i,j\in\{1,2\}$, $i\ne j$.
Now identity $F=GE$ implies that
\[
F_1=G_{11},\quad F_2=G_{21},
\]
that is, $F$ is extendable to an invertible element $G\in (A\widehat{\otimes}_\varepsilon L(X_1\oplus X_2))_0^{-1}$.

Thus  from (B) and the Bochner-Phillips-Allan-Markus-Sementsul theory, see, e.g., \cite[Th.\,9.2.7]{Ni},
we obtain the following partial extension of \cite[Th.\,1.4]{Br2}.\smallskip

\noindent (B$'$)  Let $U\subset\mathfrak M(A)$ be a dense subset and $\xi\in U$.
Suppose $F\in A\widehat{\otimes}_\varepsilon L(Y_1,Y_2)$, where $Y_i$, $i=1,2$, are complex Banach spaces, is such that for each 
$\lambda\in U$ there exists a left inverse $G_\lambda$ of $P_\lambda (F)\in L(Y_1,Y_2)$ satisfying 
\begin{equation}\label{bpams}
\sup_{\lambda\in U}\|G_\lambda\|<\infty.
\end{equation}
Let $Y:={\rm ker}\, G_\xi$.
 Assume that $GL(Y)$ is connected. Then {\em
there exist elements $H\in A\widehat{\otimes}_\varepsilon L(Y_1\oplus Y, Y_2)$, $G\in A\widehat{\otimes}_\varepsilon L(Y_2, Y_1\oplus Y)$ such that for all $\nu\in\mathfrak M(A)$}
\[
P_\nu(H) P_\nu(G)=I_{Y_2},\qquad P_\nu(G) P_\nu(H)=I_{Y_1\oplus Y}\qquad {\rm and}\qquad P_\nu(H)|_{Y_1}=P_\nu(F).
\]
\begin{R}\label{rem1.8}
{\rm Due to \cite[Th.\,9.2.7]{Ni} condition \eqref{bpams} implies that there exists an element $\widetilde F\in A\widehat{\otimes}_\varepsilon L(Y_2,Y_1)$ such that $P_\nu(\widetilde F)P_\nu(F)=I_{Y_1}$ for all $\nu\in\mathfrak M(A)$.

\noindent Statement (B$'$)  is obtained from (B) if we set $X_1:=Y_1$, $X_2:=Y$ and $X:=Y_2$.
 }
\end{R}

Note that group $GL(Y)$ is connected in the following cases (see, e.g.,  \cite[Cor.\,]{Br3} for the references): (1) ${\rm dim}_{\Co}Y<\infty$; (2) $Y_2$ is isomorphic to a Hilbert space or $c_0$ or one of the spaces $\ell^p$, $1\le p\le\infty$; (3) $Y_2$ is isomorphic to one of the spaces $L^p[0,1]$, $1< p<\infty$, or $C[0,1]$ and $Y_1$ is not isomorphic to $Y_2$.
}
\end{E}

\sect{Proofs of Theorem \ref{main} and Corollary \ref{cor1}}
\subsect{} This part contains some auxiliary results used in the proof of the theorem.
\begin{Proposition}\label{prop2.1}
The embedding $A\otimes B\hookrightarrow \mathfrak A$ extends to a morphism of Banach algebras $S: A\widehat\otimes_\pi B\rightarrow A\widehat\otimes_{\mathfrak A} B$.
\end{Proposition}
\begin{proof}
By the definition of norms on $A\widehat\otimes_\pi B$ and $\mathfrak A$ we obtain
for every $u\in A\otimes B$
\[
\|u\|_{\mathfrak A}\le {\rm inf}\left(\sum_{i=1}^n \|a_i\|_{\mathfrak A}\cdot\|b_i\|_{\mathfrak A}\right)=:\|u\|_{A\widehat\otimes_\pi B};
\]
here infimum is taken over all presentations $ u=\sum_{i=1}^n a_i b_i$,
$a_i\in A$, $b_i\in B$, $n\in\N$. 

This gives the required statement.
\end{proof}
Let $M$ be a complex Banach submanifold of $B$. We set
\begin{equation}\label{eq3}
A^M=\{u\in A\widehat\otimes_{\mathfrak A} B\, :\, P_{\xi}(u)\in M,\, \xi\in\mathfrak M(A) \}.
\end{equation}
Then $P$ maps $A^M$ into $C(\mathfrak M(A),M)$.

Let $\widehat P_\xi: A\widehat\otimes_\pi B\rightarrow B$, $\xi\in\mathfrak M(A)$, be the multiplicative projection defined by \eqref{eq2} for $\mathfrak A=A\widehat\otimes_\pi B$
and let $\widehat P: A\widehat\otimes_\pi B \rightarrow C(\mathfrak M(A),B)$, $\widehat P(u)(\xi):=\widehat P_\xi (u)$,  be the corresponding  morphism of Banach algebras.
Clearly, 
\begin{equation}\label{eq4}
\widehat P=P\circ S.
\end{equation}

We set 
\begin{equation}\label{eq5}
\widehat A^M:=\{u\in A\widehat\otimes_\pi B\, :\, \widehat P_{\xi}(u)\in M,\, \xi\in\mathfrak M(A) \}.
\end{equation}
Then due to \eqref{eq4}
\[
\widehat A^M=S^{-1}(A^M).
\]
\begin{Proposition}\label{te1}
Suppose that an open subset $M\subset B$ is
a discrete union of complex Banach homogeneous spaces. Then the space $A^M$ is locally path connected and the map $P_*: [A^M]\rightarrow [\mathfrak M(A),M]$ induced by $P$ is a bijection.
 \end{Proposition}

\begin{proof}
The fact that $A^M$ is locally path connected is obvious (because $M$ is open). The same is true for $\widehat A^M$.
 Let $u\in A^M$. Since $A\otimes B$ is dense in $A\widehat\otimes_{\mathfrak A} B$, there exists $u'\in A^M\cap  (A\otimes B)$ which belongs to the same connected component of $A^M$ as $u$. Since $S|_{A\otimes B}$ is one-to one,  the element $S^{-1}(u')$ belongs to $\widehat A^M\cap  (A\otimes B)$.
 This shows that the map $S_*: [\widehat A_M]\rightarrow [A_M]$ induced by $S$ is surjective. 
 
Further, according to \cite[Thm.\,3.9]{R}, the map $\widehat P_*: [\widehat A^M]\rightarrow  [\mathfrak M(A),M]$ induced by $\widehat P$ is bijective. Since by \eqref{eq4} $\widehat P_*=P_*\circ S_*$,
the map $S_*$ is injective and hence is bijective. These imply that $P_*$ is bijective.
 \end{proof}
 \begin{Proposition}\label{prop2.3}
It is true that
\[
A^{B^{-1}}=(A\widehat\otimes_{\mathfrak A} B)^{-1}.
\]
\end{Proposition}
\begin{proof}
Let $u\in A^{B^{-1}}$, i.e., $P(u)\in C(\mathfrak M(A),B^{-1})$. Then due to the Bochner-Phillips-Allan-Markus-Sementsul theory, see, e.g., \cite[Th.\,9.2.7]{Ni}, $u\in (A\widehat\otimes_{\mathfrak A} B)_l^{-1}$, i.e., $vu=1$ for some 
$v\in A\widehat\otimes_{\mathfrak A} B$. Hence, $uv\in {\rm id}(A\widehat\otimes_{\mathfrak A} B)$ and since
\[
\begin{array}{l}
\displaystyle
P_\xi(uv)=P_\xi(u)P_\xi(v)=P_\xi(u)P_\xi(v)P_\xi(u) P_\xi(u)^{-1}=P_\xi(u)P_\xi(vu) P_\xi(u)^{-1}\medskip\\
\displaystyle\qquad\quad\ =P_\xi(u)P_\xi(1) P_\xi(u)^{-1}=P_\xi(u) P_\xi(u)^{-1}=1\quad {\rm for\ all}\quad \xi\in\mathfrak M(A),
\end{array}
\]
applying again \cite[Th.\,9.2.7]{Ni} we find an element $w\in A\widehat\otimes_{\mathfrak A} B$ such that $w(uv)=1$. This implies 
\[
1=w(uv)=w(uv)(uv)=uv.
\]
Thus $vu=uv=1$, i.e., $u\in (A\widehat\otimes_{\mathfrak A} B)^{-1}$. This proves that
$A^{B^{-1}}\subset(A\widehat\otimes_{\mathfrak A} B)^{-1}$. The converse implication
$(A\widehat\otimes_{\mathfrak A} B)^{-1}\subset A^{B^{-1}}$ is obvious. This completes the proof of the proposition.
\end{proof}

\begin{proof}[{\bf 2.2.} Proof of Theorem \ref{main}]
(1) First we prove the theorem for $X=(A\widehat\otimes_{\mathfrak A} B)^{-1}$ and $Y=C(\mathfrak M(A),B^{-1})\, (=(C(\mathfrak M(A),B)^{-1})$.
To this end we prove that {\em for every $n\ge 0$ the map of the homotopy groups $P_*: \pi_n((A\widehat\otimes_{\mathfrak A} B)^{-1})\rightarrow \pi_n(C(\mathfrak M(A),B^{-1}))$ induced by $P$ is a bijection}. 

Indeed, according to \cite[Thm.\,4.10]{D} for every $n\ge 0$ the  map of the homotopy groups
$\widehat P_*:\pi_n((A\widehat \otimes_{\pi} B)^{-1})\rightarrow \pi_n(C(\mathfrak M(A),B^{-1}))$ induced by $\widehat P$ is an isomorphism. Since $\widehat P_*=P_*\circ S_*$, see \eqref{eq4}, the map $S_*:\pi_n((A\widehat\otimes_{\pi} B)^{-1})\rightarrow \pi_n((A\widehat\otimes_{\mathfrak A} B)^{-1})$  induced by $S$ is an injection. Hence, to prove the required result it suffices to show that $S_*$ is a surjection. 

By definition, 
\begin{equation}\label{eq7}
\begin{array}{l}
\displaystyle
\pi_n((A\widehat\otimes_{\mathfrak A} B)^{-1})=[\mathbb S^n,(A\widehat\otimes_{\mathfrak A} B)^{-1}]=[C(\mathbb S^n,(A\widehat\otimes_{\mathfrak A} B)^{-1})],\medskip\\
\displaystyle
\pi_n((A\widehat\otimes_{\pi} B)^{-1})=[\mathbb S^n,(A\widehat\otimes_{\pi} B)^{-1}]=[C(\mathbb S^n,(A\widehat\otimes_{\pi} B)^{-1})];
\end{array}
\end{equation}
here $\mathbb S^n\subset\mathbb R^n$ is the unit sphere.

In turn, $C(\mathbb S^n,A\widehat\otimes_{\mathfrak A} B)=C(\mathbb S^n)\widehat\otimes_\varepsilon (A\widehat\otimes_{\mathfrak A} B)$. Since $(A\widehat\otimes_{\mathfrak A} B)^{-1}$ is open, the latter implies that for each $u\in C(\mathbb S^n,(A\widehat\otimes_{\mathfrak A} B)^{-1})$ there exists $u'\in C(\mathbb S^n,(A\widehat\otimes_{\mathfrak A} B)^{-1}\cap (A\otimes B))$ which belongs to the same connected component of $C(\mathbb S^n,(A\widehat\otimes_{\mathfrak A} B)^{-1})$ as $u$.  Since $S|_{A\otimes B}$ is one-to-one, there exists a map $v\in C(\mathbb S^n,A\widehat\otimes_{\pi} B)$ such that $S(v(x))=u'(x)$ for all $x\in\mathbb S^n$.
Let us show that $v\in C(\mathbb S^n,(A\widehat\otimes_{\pi} B)^{-1})$. 
To this end, we have  to check that $v(x)\in (A\widehat\otimes_{\pi} B)^{-1}$ for each $x\in\mathbb S^n$. 
In fact, $\widehat P(v(x))=P(u'(x))\in C(\mathfrak M(A),B^{-1})$ (because $u'(x)\in (A\widehat\otimes_{\mathfrak A} B)^{-1}$). Then Proposition \ref{prop2.3} for $\mathfrak A=A\widehat\otimes_\pi B$ implies that $v(x)\in (A\widehat\otimes_{\pi} B)^{-1}$.
Therefore $v\in C(\mathbb S^n,(A\widehat\otimes_{\pi} B)^{-1})$ and by the definition $S_*$ maps the connected component of $C(\mathbb S^n,(A\widehat\otimes_{\pi} B)^{-1})$ containing $v$ to the connected component of $C(\mathbb S^n,(A\widehat\otimes_{\mathfrak A} B)^{-1})$ containing $u$. Hence, $S_*:\pi_n((A\widehat\otimes_{\pi} B)^{-1})\rightarrow \pi_n((A\widehat\otimes_{\mathfrak A} B)^{-1})$ is a surjection and so a bijection.

Thus we have proved that $P_*: \pi_n((A\widehat\otimes_{\mathfrak A} B)^{-1})\rightarrow \pi_n(C(\mathfrak M(A),B^{-1}))$ is a bijection for every $n\ge 0$, i.e., $P$
induces a weak homotopy equivalence of $(A\widehat\otimes_{\mathfrak A} B)^{-1}$ and $C(\mathfrak M(A),B^{-1})$.
Then according to \cite[Thm.\,15]{P} the map $P:
(A\widehat\otimes_{\mathfrak A} B)^{-1}\rightarrow C(\mathfrak M(A),B^{-1})$ is a homotopy equivalence. 

This completes the proof of the theorem in this case.\medskip

\noindent (2) Now we  prove the theorem for the spaces $X=(A\widehat\otimes_{\mathfrak A} B)_l^{-1}$ and $Y=C(\mathfrak M(A),B_l^{-1})\, (=(C(\mathfrak M(A),B)_l^{-1})$.
To this end, first we prove the following result.\smallskip

\noindent (*) {\em The map
$P_*:[(A\widehat\otimes_{\mathfrak A} B)_l^{-1}]\rightarrow [\mathfrak M(A),B_l^{-1}]$ induced by $P$ is a bijection.}\smallskip

We require
\begin{Lm}\label{lem2.4}
It is true that
\[
A^{B_l^{-1}}=(A\widehat\otimes_{\mathfrak A} B)_l^{-1}.
\]
\end{Lm}
\begin{proof}
The result follows directly from the Bochner-Phillips-Allan-Markus-Sementsul theory, see, e.g., \cite[Th.\,9.2.7]{Ni}.
\end{proof}
Now (*) follows from Proposition \ref{te1} based on Lemma \ref{lem2.4} and \cite[Prop.\,4.7]{Br4}.

Next, using (*) we show that \smallskip

\noindent (**) {\em for every $n\ge 0$ the map $P_*: \pi_n((A\widehat\otimes_{\mathfrak A} B)_l^{-1})\rightarrow \pi_n(C(\mathfrak M(A),B_l^{-1}))$ induced by $P$ is a bijection}. \smallskip

The fact that  for every $n\ge 0$ the map  $\widehat P_*: \pi_n((A\widehat\otimes_{\pi} B)_l^{-1})\rightarrow \pi_n(C(\mathfrak M(A),B_l^{-1}))$ induced by $\widehat P$ is a bijection follows by repeating line by line the argument of \cite[Cor.\,4.2]{R} based on associativity of the projective tensor product, where instead of Corollary 4.1 of \cite{R} we use our statement (*). Then (**) can be deduced from here as in the proof of part (1) of the theorem. We leave the details to the readers.

Since due to (**) the map $P$ induces a weak homotopy equivalence of $(A\widehat\otimes_{\mathfrak A} B)_l^{-1}$ and $C(\mathfrak M(A),B_l^{-1})$, it is a homotopy equivalence of these spaces according to \cite[Thm.\,15]{P}.

The proof of the theorem in this case is complete.\medskip

\noindent (3) Finally, we prove the theorem for spaces $X={\rm id}(A\widehat\otimes_{\mathfrak A}B)$ and $Y=C(\mathfrak M(A), {\rm id}\, B)\, (={\rm id}(C(\mathfrak M(A),B))$.

We use the following result.
\begin{Lm}\label{lem2.5}
The map $\widehat P: {\rm id}(A\widehat\otimes_{\pi}B)\rightarrow C(\mathfrak M(A), {\rm id}\, B)$ is a homotopy equivalence.
\end{Lm}
\begin{proof}
The proof repeats line by line the argument of  \cite[Cor.\,4.2]{R} based on associativity of the projective tensor product, where instead of Corollary 4.1 of \cite{R} one uses \cite[Thm.\,4.5]{R}. We leave the details to the readers.
\end{proof}
Since $\widehat P=\widehat P\circ S$, Lemma \ref{lem2.5} implies that for every $n\ge 0$ the map $S_*:\pi_n({\rm id}(A\widehat\otimes_{\pi}B))\rightarrow \pi_n({\rm id}(A\widehat\otimes_{\mathfrak A}B))$ induced by $S$ is an injection. Let us show that it is a surjection  as well. To this end we will use a construction from \cite[Prop.\,4.3]{R}:

Let $\mathcal A$ be a complex Banach algebra with identity $1$. Then if $a\in {\rm id}\,\mathcal A$ the spectrum $\sigma(a)$ of $a$ in $\mathcal A$ is contained in the set $\{0,1\}\subset\mathbb C$. Let 
\[
D=\{z\in\mathbb C\, :\, |z|<\mbox{$\frac{1}{4}$}\}\cup\{z\in\mathbb C\, :\, |z-1|<\mbox{$\frac 1 4$}\}.
 \]
 Let $U_{\mathcal A}$ be an open neighbourhood of ${\rm id}\,\mathcal A$ such that $a\in U_{\mathcal A}$ implies $\sigma(a)\subset D$. Let 
 $\Gamma_0=\{z\in\mathbb C\, :\, |z|=\mbox{$\frac{1}{4}$}\}$, $\Gamma_1=\{z\in\mathbb C\, :\, |z-1|=\mbox{$\frac 1 4$}\}$, and let $F$ be the holomorphic function defined on $D$ equals $0$ on $\{z\in\mathbb C\, :\, |z|<\frac{1}{4}\}$ and $1$ on $\{z\in\mathbb C\, :\, |z-1|<\frac 1 4\}$. 
 \begin{Proposition}[\cite{R}, Proposition 4.3]\label{prop2.6}
 The map $r_{\mathcal A}:U_\mathcal A\rightarrow {\rm id}\,\mathcal A$,
 \begin{equation}\label{eq2.5}
 r_\mathcal A(a):=\frac{1}{2\pi i}\int_{\Gamma_0\cup\Gamma_1}F(\xi)(\xi-a)^{-1}\,d\xi=
 \frac{1}{2\pi i}\int_{\Gamma_1}(\xi-a)^{-1}\,d\xi,\quad a\in U_\mathcal A,
 \end{equation}
 is a holomorphic retraction of $U_{\mathcal A}$ onto ${\rm id}\,\mathcal A$.
 \end{Proposition}

Going back to our setting we prove the following result.
\begin{Lm}\label{lem2.7}
We have 
\begin{equation}\label{eq2.6}
\begin{array}{l}
\displaystyle
U_{C(\mathfrak M(A),B)}=C(\mathfrak M(A),U_B),\quad U_{A\widehat\otimes_{\mathfrak A}B}=P^{-1}(U_{C(\mathfrak M(A),B)}),\medskip\\ 
\displaystyle U_{A\widehat\otimes_{\pi}B}=S^{-1}(U_{A\widehat\otimes_{\mathfrak A}B}).
\end{array}
\end{equation}
Moreover,
\begin{equation}\label{eq2.7}
\begin{array}{l}
\displaystyle
(r_{C(\mathfrak M(A),B)}(f))(x)=r_B(f(x)),\quad f\in U_{C(\mathfrak M(A),B)},\ x\in \mathfrak M(A);\\
\\
\displaystyle
P\circ r_{A\widehat\otimes_{\mathfrak A}B}=r_{C(\mathfrak M(A),B)}\circ P,\qquad S\circ r_{A\widehat\otimes_{\pi}B}= r_{A\widehat\otimes_{\mathfrak A}B}\circ S.
\end{array}
\end{equation}
\end{Lm}
\begin{proof}
Let $f\in C(\mathfrak M(A),B)$. It is easily seen (because $\mathfrak M(A)$ is compact) that
\begin{equation}\label{eq2.9}
\sigma(f)=\bigcup_{x\in\mathfrak M(A)}\sigma(f(x)).
\end{equation}
This implies that $f\in U_{C(\mathfrak M(A),B)}$, i.e., $\sigma(f)\subset D$, if and only if $\sigma(f(x))\subset D$ for all $x\in\mathfrak M(A)$, i.e.,  $f\in C(\mathfrak M(A), U_B)$.
This proves the first identity of \eqref{eq2.6}.

Next, it follows from Proposition \ref{prop2.3} that
\[
\sigma(a)=\sigma(S(a))=\sigma(\widehat P(a))\quad {\rm for\ all}\quad a\in A\widehat\otimes_{\pi}B.
\]
These imply the second and the third identities of \eqref{eq2.6}.

Now, due to \eqref{eq2.5} and \eqref{eq2.6} for $f\in U_{C(\mathfrak M(A),B)}$,  $x\in \mathfrak M(A)$
\[
(r_{C(\mathfrak M(A),B)}(f))(x)=\frac{1}{2\pi i}\int_{\Gamma_1}(\xi-f)^{-1}(x)\,d\xi=\frac{1}{2\pi i}\int_{\Gamma_1}(\xi-f(x))^{-1}\,d\xi=:r_B(f(x)).
\]
This gives the first identity of \eqref{eq2.7}. Proofs of the remaining identities are similar as $P$ and $S$ are morphisms of Banach algebras.
\end{proof}

Using Lemma \ref{lem2.7} let us prove that $S_*:\pi_n({\rm id}(A\widehat\otimes_{\pi}B))\rightarrow \pi_n({\rm id}(A\widehat\otimes_{\mathfrak A}B))$, $n\in\mathbb Z_+$, is a surjection.

Let $f\in C(\mathbb S^n, {\rm id}(A\widehat\otimes_{\mathfrak A}B))$. Since $C(\mathbb S^n, A\widehat\otimes_{\mathfrak A} B)=C(\mathbb S^n)\widehat\otimes_{\varepsilon}(A\widehat\otimes_{\mathfrak A} B)$, there exists a map $g\in C(\mathbb S^n, U_{A\widehat\otimes_{\mathfrak A} B}\cap (A\otimes B))$ such that $tf+(1-t)g\in C(\mathbb S^n, U_{A\widehat\otimes_{\mathfrak A} B})$ for all $t\in [0,1]$. In particular, $r_{A\widehat\otimes_{\mathfrak A} B}(tf+(1-t)g)\in C(\mathbb S^n, {\rm id}(A\widehat\otimes_{\mathfrak A} B))$, $t\in [0,1]$, is a homotopy between $f$ and $r_{A\widehat\otimes_{\mathfrak A} B}(g)$. Hence, these maps determine the same element of $\pi_n({\rm id}(A\widehat\otimes_{\mathfrak A}B))$. 

Further, since $S|_{A\otimes B}$ is one-to-one, there exists a map $g'\in C(\mathbb S^n, A\widehat \otimes_\pi B)$ such that $S(g')=g$. Then according to the third identity of equation \eqref{eq2.6}, $g'\in C(\mathbb S^n, U_{A\widehat \otimes_\pi B})$. In turn, the third identity of equation \eqref{eq2.7} implies that
\[
 S(r_{A\widehat\otimes_{\pi}B}(g'))= (r_{A\widehat\otimes_{\mathfrak A}B}\circ S)(g')=r_{A\widehat\otimes_{\mathfrak A} B}(g).
\]
The latter and the definition of $g$ show that $S_*$ maps the homotopy class
of $r_{A\widehat\otimes_{\pi}B}(g')$ in $\pi_n({\rm id}(A\widehat\otimes_\pi B))$ to
the homotopy class of $f$ in $\pi_n({\rm id}(A\widehat\otimes_\mathfrak A B))$.
This proves surjectivity of $S_*$ and, hence, its bijectivity.

In turn, this and Lemma \ref{lem2.5} imply that $P: {\rm id}(A\widehat\otimes_{\mathfrak A}B)\rightarrow C(\mathfrak M(A), {\rm id}\, B)$ is a weak homotopy equivalence and, hence, a homotopy equivalence due to \cite[Thm.\,15]{P}.

The proof of the theorem is complete.
\end{proof}

\begin{proof}[{\bf 2.3.} Proof of Corollary \ref{cor1}]
First, let us prove the required result for the map $\{P\}_l:\{(A\widehat\otimes_{\mathfrak A}  B)_l^{-1},\}\rightarrow \{(C(\mathfrak M(A), B))_l^{-1}\}$. 

If $f\in C(\mathfrak M(A), B))_l^{-1}$, then  by Theorem \ref{main} there is $u\in (A\widehat\otimes_{\mathfrak A}  B)_l^{-1}$ such that $P(u)$ belongs to the same connected component as $f$. Since each connected component of $(C(\mathfrak M(A), B))_l^{-1}$ is a complex Banach homogeneous space with respect to the action of $(C(\mathfrak M(A), B^{-1}))_0$, see \cite[Porp.\,4.7]{Br4}, there exists $g\in (C(\mathfrak M(A), B^{-1}))_0$ such that $gP(u)=f$. Thus $\{P\}_l$ maps the equivalence class of $u$ to the equivalence class of $f$; hence, $\{P\}_l$ is a surjection. 

Further, suppose that  $u_1,u_2\in (A\widehat\otimes_{\mathfrak A}  B)_l^{-1}$ are such that
$P(u_2)=gP(u_1)$ for some $g\in C(\mathfrak M(A), B^{-1})$. According to Theorem \ref{main}, there exists $h\in (A\widehat\otimes_{\mathfrak A}  B)^{-1}$ such that $P(h)$ belongs to the same connected component as $g$. Since $P(h^{-1}u_2)=P(h^{-1})P(u_2)=
(P(h^{-1})g)P(u_1)$, $P(h^{-1}u_2)$ and $P(u_1)$ belong to the same connected component. Hence, due to Theorem \ref{main} $h^{-1}u_2$ and $u_1$ are in the same connected  component as well. Thus by \cite[Porp.\,4.7]{Br4} there exists $h'\in ((A\widehat\otimes_{\mathfrak A}  B)^{-1})_0$ such that $h'(h^{-1}u_2)=u_1$, i.e., $u_2\sim u_1$. This shows that $\{P\}_l$ is an injection.

The proof of the corollary for the map $\{P\}_{id}:\{{\rm id}(A\widehat\otimes_{\mathfrak A}  B)\}\rightarrow \{{\rm id}(C(\mathfrak M(A), B))\}$ is similar, cf. \cite[Cor.\,4.7]{R}.
\end{proof}

\sect{Proofs}
All proofs presented in this section follow (with minor modifications) the proofs of some known results. Thus we only sketch them referring the readers to the corresponding papers. For basic results of homotopy and fibre bundles theories  used in the proofs, see, e.g., \cite{Hu1}, \cite{Hus}.
\begin{proof}[Proof of Theorem \ref{te1.4}]
Since $A\in \mathscr C$,  every continuous map of $\mathfrak M(A)$ into an absolute neighbourhood retract $Y$ is homotopic to a constant map into $Y$, see the argument of the proof of \cite[Thm.\,1.2]{BRS}. In particular, this is valid for continuous maps of $\mathfrak M(A)$ into $Y\in\bigl\{  B^{-1},  B_l^{-1}, {\rm id}\, B\bigr\}$. Using this and repeating literally the arguments of the proofs of \cite[Thm.\,5.1,\,Thm.\,5.3]{BRS} one obtains the required results of Theorem \ref{main}.
\end{proof}
\begin{proof}[Proof of Corollary \ref{cor1.5}]
Let $X\in\bigl\{ (A\widehat\otimes_{\mathfrak A}  B)_l^{-1}, {\rm id}(A\widehat\otimes_{\mathfrak A}  B)\bigr\}$ and $u\in X$.  According to Theorem \ref{te1.4} $u$ belongs to the same connected component of $X$ as $P_\xi(u)$. This gives the required result because each connected component of $X$ is a complex Banach homogeneous space under the corresponding action of the group $(A\widehat\otimes_\mathfrak A B)_0^{-1}$.
\end{proof}
\begin{proof}[Proof of Theorem \ref{teo4.1}]
According to Theorem \ref{main} it suffices to prove the result for spaces $C(\mathfrak M(A),Y)$, where $Y\in\{B_l^{-1}, {\rm id}\, B\}$, i.e., to show that under the hypothesis of the theorem $u\in C(\mathfrak M(A),Y)$ and the constant map
$u(\xi)$ belong to the same connected component of $C(\mathfrak M(A),Y)$.
Since $C(\mathfrak M(A),B)=C(\mathfrak M(A))\widehat\otimes_\varepsilon B$, where $\mathfrak M(A)$ is the inverse limit of an inverse limiting system of compact connected spaces  (with trivial second \v{C}ech cohomology groups homotopy equivalent to compact spaces of covering dimension $\le 2$) and $Y$ is an absolute neighbourhood retract, each  $u\in C(\mathfrak M(A),Y)$ is homotopy equivalent to a map $u' $ which is the pullback to $\mathfrak M(A)$ of a map $\tilde u\in C(Z,Y)$  for some space $Z$ of the inverse limiting system, see, e.g., 
\cite[Ch.\,X,\,Thm.\,11.9]{ES}. Thus $u$ and $u'$ belong to the same connected component of the Banach homogeneous space $C(\mathfrak M(A),Y)$ and, moreover, the stabilizers in $B_0^{-1}$ of points $u(\lambda)$, $u'(\lambda)$ and $\tilde u(z)$, $\lambda\in\mathfrak M(A)$, $z\in Z$,  are isomorphic and, hence, due to the hypothesis are connected, see, e.g., \cite[Prop.\,1.4]{R}. This shows that it suffices to prove the theorem for maps $\tilde u\in C(Z,Y)$.

Next, as in the proof of \cite[Thm.\,4.4\,(b),\,Thm.\,4.8\,(a)]{Br5} the result will be proved if we will show that each principal bundle on $Z$ whose fibre is a {\em connected} complex Banach Lie group is trivial, see \cite[Thm.\,7.1]{Br5}. Since $Z$ is homotopy equivalent to a compact connected space $\widetilde Z$ such that 
${\rm dim}\,\widetilde Z\le 2$ and $H^2(\widetilde Z,\Z)=0$, it suffices to prove the claim for principal bundles over $\widetilde Z$. In turn, following the argument of \cite[Prop.\,7.2]{Br5} one reduces the problem to bundles with simply connected fibres.
Thus we have to prove that {\em if $\pi: E\rightarrow \widetilde Z$ is a principal bundle whose fibre is a simply connected complex Banach Lie group $G$, then $E$ is trivial}, i.e., admits a continuous map (section) $s:\widetilde Z\rightarrow E$
such that $\pi\circ s={\rm id}_{\widetilde Z}$.

To this end, suppose that $(V_i)_{1\le i\le k}$ is a finite cover of $\widetilde Z$ by compact subsets such that $E$ is trivial over each $V_i$ (existence of the cover follows from local triviality of a principal bundle). We set $V^i=\cup_{j=1}^i V_j$ and prove by induction on $i$ that $E$ is trivial over each $V^i$. 

For $i=1$ the statement is trivial. Assuming that it is valid for all $j<i$ let us prove it for $i$. 

Indeed, since $E$ is trivial over $V^{i-1}$ and $V_i$ it has continuous 
sections $s^{i-1}: V^{i-1}\rightarrow E$ and $s_i: V_i\rightarrow E$. Then
there is $g_{i}\in C(V^{i-1}\cap V_i, G)$ such that 
\begin{equation}\label{eq3.1}
s^{i-1}=g_i\cdot s_i\quad {\rm on}\quad V^{i-1}\cap V_i.
\end{equation}
(Here $G\times E\rightarrow E$, $(g,e)\mapsto g\cdot e$, is the continuous action of $G$ on fibres of $E$.)

\noindent Without loss of generality we may assume that $V^{i-1}\cap V_i\ne\emptyset$ (for otherwise the statement is obvious).  Since $G$ is simply connected, the fundamental group $\pi_1(G)$ is trivial and so the \v{C}ech cohomology group $H^1(V^{i-1}\cap V_i,\pi_1(G))=0$.  Then since ${\rm dim}\,\widetilde Z\le 2$, \cite[(10.5)]{Hu2} implies that $g_i$ admits an extension $\tilde g_i\in C(V_i,G)$. We set $\tilde s_i:=\tilde g_i\cdot s_i$. Then $\tilde s_i$ is a continuous section of $E$ over $V_i$ and due to \eqref{eq3.1}
$s^{i-1}=\tilde s_i$ on $V^{i-1}\cap V_i$. Thus these sections glue together to determine a continuous section of $E$ over $V^i:=V^{i-1}\cup V_i$. Hence, $E$ is trivial over $V^i$. This completes the proof of the induction step.

Taking here $i=k$ we obtain that $E$ is trivial on $\widetilde Z$.
As it was explained above, the latter gives the required result.
\end{proof}

\end{document}